\documentclass[11pt,openany]{article}
\usepackage[utf8]{inputenc}
\usepackage[english]{babel}
\usepackage[margin=0.75in]{geometry}
\usepackage{color}
\usepackage{enumerate}
\usepackage{amsmath, amsthm, amssymb, amsfonts}
\usepackage{mathtools}
\usepackage[all,cmtip]{xy}
\usepackage{todonotes}
\usepackage{hyperref}
\usepackage{todonotes}
\allowdisplaybreaks
\newtheorem{theorem}{Theorem}[section]

\newtheorem{proposition}[theorem]{Proposition}
\theoremstyle{definition}
\newtheorem{definition}[theorem]{Definition}
\newtheorem{example}[theorem]{Example}

\newtheorem{calc'}{Calculation}

\newtheorem*{claim*}{Claim}

\newtheorem{case'}{Case}
\newtheorem{case''}{Case}

\newtheorem*{conjecture*}{Conjecture}

\newenvironment{theorem*}[2][Theorem]{\begin{trivlist}
\item[\hskip \labelsep {\bfseries #1}\hskip \labelsep {\bfseries #2}]}{\end{trivlist}}
\newenvironment{lemma*}[2][Lemma]{\begin{trivlist}
\item[\hskip \labelsep {\bfseries #1}\hskip \labelsep {\bfseries #2}]}{\end{trivlist}}
\newenvironment{corollary*}[2][Corollary]{\begin{trivlist}
\item[\hskip \labelsep {\bfseries #1}\hskip \labelsep {\bfseries #2}]}{\end{trivlist}}
\begin{document}
\title{\texorpdfstring{Zeta functions of projective hypersurfaces with ADE Singularities}{Zeta functions of projective hypersurfaces with ADE Singularities}}
\author{Matthew Cheung \\ Mathematics, UCI, Irvine, CA, USA\\ {\color{blue}\href{mailto:matthc4@uci.edu}{matthc4@uci.edu}}} 
\date{}

\maketitle

\begin{abstract}

Given a hypersurface, $X$, prime $p$, the zeta function is a generating function for the number of $\mathbb{F}_{p}$ rational points of $X$. Scott Stetson and Vladimir Baranovsky provided an algorithm with Mathematica for the ordinary double point case. In this paper, I explain the necessary steps to extend Scott Stetson's and Vladimir Baranovsky's algorithm to hypersurfaces with ADE singularities in 3-dimensional projective space. In the process of doing so, I characterized the Jacobian ideal by being annihilated by a set of differential operators.

\end{abstract}

\tableofcontents

\bigskip

\section{Introduction}\label{sec0}
In the past, the zeta functions for smooth surfaces have been computed. One method involves the deformation method given in Lauder [11] and [12]. A second method involves computing the action of Frobenius and reducing in cohomology--see Costa, Harvey, and Kedlaya [2] and Rybakov [14]. A Sage code for computing smooth surfaces is given by Sperber and Voight [18]. Stetson and Baranovsky [17] extended zeta function computation to hypersurfaces with ordinary double points, and Scott Stetson created a non-automated code in Mathematica computing the zeta functions of hypersurfaces with ordinary double points. I will generalize the theory from Stetson and Baranovsky [16] and give an algorithm for computing the zeta functions of hypersurfaces with ADE singularities. Furthermore, I will provide a non-automated Sage code for computing such zeta functions. \\

\begin{definition}\label{Definition 1.1} Given a projective hypersurface $X$ defined by equation $f(w:x:y:z)$ in $\mathbb{P}^{3}$, the affine cone over $X$ is $f(w:x:y:z)$ viewed as a function $f(w:x:y:z)$ over $\mathbb{C}^{4}$. \end{definition} \begin{definition}. Let $K$ be a field or $\mathbb{Z}_{p}$. A polynomial over $K^{3}$ has an ADE singularity at the origin if locally after a formal analytic change of coordinates (i.e., change of coordinates is a formal power series with coefficients in $K$), the function is one of the following \\
	\begin{align*}	
		&A_{n}: w^{2}+x^{2}+y^{n+1}=0 \\
		&D_{n}: w^{2}+y(x^{2}+y^{n-2})=0\\
		&E_{6}: w^{2}+x^{3}+y^{4}=0 \\
		&E_{7}: w^{2}+x(x^{2}+y^{3})=0\\
		&E_{8}: w^{2}+x^{3}+y^{5}=0.
	\end{align*}
\end{definition}
Let $p$ be a prime and $q= p^{n}$ for some $n$. Another definition in the case $K=\mathbb{F}_{p}$ for ADE singularities over finite fields is given in Greuel [8].  
 Let $N_{r} =\vert X(\mathbb{F}_{p^{r}}) \vert$. Let $X$ be a hypersurface in $\mathbb{P}^{3}$ with ADE singularities over $\mathbb{F}_{p}$.  Then the zeta function of a variety $X$ over $\mathbb{F}_{q}$ is a generating function for $N_{r}$ given by 
\[ Z(X,t)=Z(t)=\exp \sum_{r=1}^{\infty}\frac{N_{r}t^{r}}{r}.\]

Let $U=\mathbb{P}^{3} - X$. Then $U$ is affine and smooth.  From Theorem 3.1 of Gerkmann [7], as we are working over $\mathbb{P}^{3}$,
$$Z(X,t)=\frac{1}{(1-t)(1-qt)(1-q^{2}t)p(t)}$$
$$p(t)=\text{det}(1-tq^{3}\text{Frob}_{q}^{-1}\vert H_{\text{rig}}^{2}(U))$$
where $H_{\text{rig}}^{2}(U)$ is the rigid cohomology group of $U$ and $Frob_{q}^{-1}$ is the action of the inverse of lifted Frobenius on the rigid cohomology groups. See Gerkmann [7] for a review of rigid cohomology. But for the purposes of this paper, we will be applying de Rham cohomology to find the basis elements of the rigid cohomology groups; hence, the need to fully understand rigid cohomology is not necessary.

\begin{definition}\label{Definition 1.3}. Let $f \in \mathbb{F}_{p}[w,x,y,z]$ with ADE singularities at $s_{1},...,s_{n}$. A lift of $f$, $\tilde{f}$, to $\mathbb{Z}_{p}[w,x,y,z]$ is equisingular if 
\begin{enumerate}
\item $\mathbb{Z}_{p}[w,x,y,z]/(\tilde{f_{w}},\tilde{f_{x}},\tilde{f_{y}},\tilde{f_{z}})$ has no $p$ torsion parts. Here, $\tilde{f_{w}},\tilde{f_{x}},\tilde{f_{y}},\tilde{f_{z}}$ represent the lifted partials.
\end{enumerate}
\end{definition}
The issue we want to avoid is for example the case of an $A_{4}$ singularity given by standard equation $uv=t^{5}$ and our prime is $p=5$. The partials are given by $v,u,5t^{4}$. Over $\mathbb{F}_{5}$, the third term is $0$. This will lead to issues in lifting the ideal $(v,u,5t^{4})$. We wish to avoid these primes. Hence, by allowing only equisingular lifts, we avoid these situations as $\mathbb{Z}_{p}[u,v,t]/(v,u,5t^{4})$ has a $5$ torsion. The term $t^{4}$ is not $0$ but $5t^{4}$ is $0$. \\

Let $f\in \mathbb{F}_{p}[x,y,z]$ be a homogeneous polynomial with ADE singularities over $\mathbb{F}_{p}$. Let $\tilde{f}$ be an equisingular lift. Let $X$ be the zero set of $f$ over $\mathbb{F}_{p}$ and $X_{\mathbb{Z}_{p}}$ be the zero set of $\tilde{f}$ over  $\mathbb{Z}_{p}$. Let $U$ and $U_{\mathbb{Z}_{p}}$ be the complements of the spaces respectively. From Hironaka [10], we can perform a sequence of blow ups at the singular points in $\mathbb{P}^{3}$ such that $X_{\mathbb{Z}_{p}}$ is resolved to a smooth divisor with normal crossings. Then from Corollary 2.6 of Baldassarri and Chiarellotto [1],
\[ H_{\text{rig}}^{i}(U)\cong H_{\text{dR}}^{i}(U_{\mathbb{Z}_{p}}) \ \ \ \ \ \ 0\leq i \leq 2\text{dim}(U)\]

Hence, we can focus on studying the de Rham cohomology of the complement.  
\section{Preliminaries}\label{sec1}

A differential $n$-form in $k^{m+1}$ is a form 
$\omega=\sum_{I}c_{I}dx_{i_{1}}\wedge...\wedge dx_{i_{n}}$
where $I=(i_{1},...,i_{n})$ and $c_{I}\in k[x_{0},...,x_{m}]$. Let $\Omega_{m}^{n}$ be the space of $n$-forms of weight $m$ where if $\omega=x_{0}^{a_{0}}...x_{s}^{a_{s}}dx_{i_{1}}\wedge ... \wedge dx_{i_{n}}$, 
\[ |\omega|= a_{0}+...+a_{s}+n.\]

From Chapter 6 of Dimca [3], every differential $k$-form $\omega$ for $k>0$ on $U$ is written as $\omega= \frac{\Delta(\gamma)}{f^{s}}$, where $\Delta$ is the contraction of the Euler vector field $\sum x_{i}\frac{\partial}{\partial x_{i}}$ and $\gamma \in \Omega_{sN}^{p+1}$ where $N=\text{deg}(f)$.  A simple calculation shows that we can express the differential of the form $\omega$ as $d\omega= \frac{\Delta(\delta)}{f^{s+1}}$ for some $\delta \in \Omega_{(s+1)N}^{p+2}$. One can calculate that $d\omega = -\frac{\Delta(fd\gamma - sdf \wedge \gamma)}{f^{s+1}}$.

\begin{definition}For $k\geq 0$, we define our differential form $d_{f} \colon \Omega^{k}\longrightarrow \Omega^{k+1}$ to be
	$$d_{f}(\omega)=fd\omega - \frac{|\omega|}{N}df\wedge \omega $$
	for homogeneous differential form $\omega$.
\end{definition}

In other words, the homogeneity condition allows us to forget denominators and work with polynomials. However, our differential is no longer the usual one since the differential is now in the form of the Koszul differential plus the de Rham differential. 

\begin{definition} Let $(B,d',d'')$ be the double complex given by $B^{s,t}=\Omega_{tN}^{s+t+1}$ where $d'=d$ and $d''(\omega)= -|\omega|N^{-1}df\wedge \omega$ for a homogeneous differential form $\omega$. \end{definition}
\begin{definition} Let $(\text{Tot}(B)^{*},D_{f})$ be the total complex given by 
	$ \text{Tot}(B)^{m}=\bigoplus_{s+t=m}B^{s,t}$ with filtration $F^{s}\text{Tot}(B)^{m}=\bigoplus_{k\geq s}B^{k,m-k}$ where $D_{f}=d'+d''$. \end{definition} 
Saito [15] shows if $m=\text{dim}_{\mathbb{C}}f^{-1}(0)$, then $H^{k}(K_{f}^{*})=0$ for $k\leq n-m$, where $\text{dim}_{\mathbb{C}}f^{-1}(0)$ is the dimension of the singular locus, $n$ corresponds to $\mathbb{P}^{n}$, which for our case is $3$, and $H^{k}(K_{f}^{*})$ is the cohomology in the vertical direction with respect to the $df\wedge$ differential. In the smooth case, $m=0$ so only the top cohomology group $H^{n+1}(K_{f}^{*})_{tN}$ is nonzero. As only one diagonal remains on the $E_{1}$ page, the de Rham differential is trivial; hence, in the smooth case, the spectral sequence degenerates at the $E_{1}$ page and converges to the cohomology of the total complex. \\

In the singular case, $m=1$ so our case involves the top and second to top cohomology of the Koszul complex. As there are two diagonals on the $E_{1}$ page, the de Rham differential need not be trivial. Let $\mu(X)$ be the global Milnor number of our hypersurface X. For the purposes of this paper, a type $A_{k}, D_{k},E_{k}$ has Milnor number $k$. The global Milnor number is defined to be the sum of all Milnor numbers. By Corollary 1.5 of Dimca and Sticlaru [6],  $H^{n}(K_{f})_{m}=\mu(X)$ for $m\geq 3(N-2)$. (The general formula is $n(N-2)$ if we are working in $\mathcal{P}^{n}$ instead of $\mathbb{P}^{3}$.)  Hence, the dimensions of the diagonals eventually stabilize to the global Milnor number. Furthermore, Theorem 2 of Saito [16] gives that for weighted homogeneous equations, the spectral sequence degenerates on the $E_{2}$ page. Equation 2.10 above Example 2.1 of Dimca and Sticlaru [5] gives that for weighted homogeneous equations, all nonzero terms on the $E_{2}$ page lie inside the first quadrant, not including the $x$-axis and $y$-axis. Using this, I constructed a Sage code computing the basis elements on the $E_{2}$ page. The code mainly involves constructing the matrix for the two differentials and using linear algebra to compute the quotient groups.\\

The next step is computing Frobenius using the basis elements on the $E_{2}$ page. Given a basis element on the $E_{2}$ page, we consider the action of Frobenius on this element. For the remainder of the text, let $\Omega=dw\wedge dx \wedge dy \wedge dz$.  Let $\frac{h\Omega}{f^{\ell}}$ be one of the basis elements. Then the action of the lifted Frobenius, $\hat{F}$, is given by 
$$\hat{F}\left(\frac{h\Omega}{f^{\ell}}\right)=p^{3}\frac{h(w^{p},x^{p},y^{p},z^{p})\prod_{i=0}^{3}x_{i}^{p-1}\Omega}{f^{p\ell}}\left(\sum_{k=0}^{\infty}p^{k}\frac{\alpha_{k}g^{k}}{f^{pk}}\right), $$ 
where $\alpha_{k}$ is the $k$-th coefficient of the power series expansion $(1-t)^{-\ell}= \alpha_{0}+\alpha_{1}t+\alpha_{2}t^{2}+...$ , 
$x_{0}=w, x_{1}=x, x_{2}=y, x_{3}=z$, and $pg=f(w,x,y,z)^{p}-f(w^{p},x^{p},y^{p},z^{p})$. Given each term in the sum, the goal is to express the image of Frobenius as a linear combination of the basis elements on the $E_{2}$ page. \\

Applying Frobenius to each basis element and reducing in cohomology, the action of Frobenius is represented as a square matrix. Note that since this is an infinite sum, we need to truncate the sum to the first $N$ summands. Gerkmann [7] goes over how far one needs to truncate, but for the purposes of Sage calculation, each additional term in the sum generally gives us accuracy up to the next power of $p$ for the examples I am doing. To be more accurate, in the examples considered in this paper, the $p$-adic expansion of the numbers in the truncated Frobenius matrix converge to the $p$-adic expansion of the actual value of the Frobenius matrix. If $k=0$ gives accuracy up to $p^{2}$, then going up to $k=1$ gives accuracy up to $p^{3}$. If one of the entries is for example $-5$, then an example of what one might see will be
\begin{align*}
	k=0 & \rightarrow 5+5^{2}+2\cdot 5^{3}+3\cdot5^{4}+... \\
	k=0, k=1 &\rightarrow 5+4\cdot 5^{2} + 3\cdot 5^{3}+2\cdot 5^{4}+... \\
	k=0, k=1, k=2 &\rightarrow 5+4\cdot 5^{2} + 4\cdot 5^{3} +\cdot 2\cdot5^{4}+...
\end{align*}
Continuing on, one will see the values converge to the $5+4\cdot 5^{2}+4\cdot 5^{3}+4\cdot 5^{4}+...,$ the p-adic expansion of $-5$. \\

While the goal is clear, there are several issues that come into play. The first issue is that reducing in cohomology takes a while if we go the direct route which is explained in the coming example. The second issue is that reduction involves using a Gr\"{o}bner basis which may get large. The third issue is that the image of Frobenius is high in degree which ties into the first and second issue. Although my code does not fix the Grobner basis issue, my code does fix the other two problems.

\begin{example} Let $f$ be degree $3$ with $p=5$ and $h_{1},h_{2},h_{3}$ be our basis elements on the $E_{2}$ page of degree 2. Then by homogeneity, $\ell=2$ in the Frobenius equation above. Now to make things simple, let us consider $k=0$. The sum goes away and on the denominator we have $f^{p\ell}=f^{10}$. Since $f$ is cubic, the denominator is degree 30. Hence, excluding the weight of $\Omega$, our image is of degree $30-4=26$. Let $u$ be the degree 26 polynomial. As this space is $0$ on $E_{2}$ page, $u$ is in the Jacobian ideal, provided we subtract off appropriate elements in the image of the de Rham differential. Let $P_{k}\Omega_{j}$ be $j-$forms whose coefficients are polynomials of degree $k$.
	
	We will need to compute the image of the $df \wedge$ map from $P_{24}\Omega_{3}$ to $P_{26}\Omega_{3}$ and quotient the whole space $P_{26}\Omega_{3}$ by the image. Note computing the quotient is in order to find the basis elements as this will give the space on the top diagonal of the $E_{1}$ page. In terms of reducing the cohomology, record the image as a matrix.  Next, we need to compute the kernel of the $df \wedge$ map $P_{27}\Omega_{3}\rightarrow P_{29}\Omega_{4}$ and the image of the $df\wedge$ map from $P_{25}\Omega_{2}\rightarrow P_{27}\Omega_{3}$ and then compute the quotient to get the space on the subdiagonal on the $E_{1}$ page. After computing the quotient, apply de Rham differential to go from $P_{27}\Omega_{3}$ to $P_{26}\Omega_{4}$. Record this image of de rham as a matrix. Combine the two recorded matrix into one matrix. Since $u$ is degree 26 and the space of degree 26 polynomials on the top diagonal is 0 on the $E_{2}$ page, when we combine the two matrices recorded above into one matrix, we get a matrix with full rank. Express $u$ as a vector and append to the matrix of full rank. Now apply linear algebra to solve for $u$ in terms of the image of $df \wedge$. Take the values for the $df \wedge$ portion. Now we can reduce in cohomology and reduce the degree of $u$ by $\text{deg}(f)$ because Koszul differential reduces the degree by $\text{deg}(f)-1$ since the partials have one degree less, and then applying de Rham differential reduces the degree by 1. We will apply the process again.  The issue here is that the space of polynomials of degree $25$ to $29$ is large. Computing the kernel and image may be simple but computing the quotient may take a long run time and after the long computation, we only reduce the degree of the image by $\text{deg}(f)$. Furthermore, this is just the $k=0$ term of the summation and only for one of our basis elements.
\end{example}

To work with lower powers in general, we decide to use the left inverse of Frobenius. Remke [13] gave in Notation 12 the formula for the left inverse of Frobenius. By this, we mean that Frobenius has a left inverse on the level of varieties. Let $\hat{F}^{-1}$ be the corresponding operator on level of cohomology. Let $\psi: A^{\dagger}\rightarrow A^{\dagger}$ be the $Q_{q}$ linear operator given by 
$$\psi(\prod x_{i}^{a_{i}})=
\begin{cases}
	\prod x_{i}^{a_{i}/p} & \text{if} \  a_{i}\equiv 0 (\text{mod} \  p) \ \  \forall i \\
	0  & \text{otherwise}
\end{cases}
$$

Note that since the action of Frobenius is taking $p$-th powers, the inverse should involve taking $p$-th roots. Taking $x_{0}=w, x_{1}=x, x_{2}=y, x_{3}=z$ and $\Delta=f(w,x,y,z)^{p}-f(w^{p},x^{p},y^{p},z^{p})$, the action of the inverse is given by
$$\hat{F}^{-1}\left(\frac{h\Omega}{f^{\ell}}\right)=\left(\sum_{k} \frac{\psi(f^{p-\ell}h\prod_{i=0}^{3}x_{i}\Delta^{i})}{f^{k+1}} \right)\frac{\Omega}{p^{3}\prod_{i=0}^{3}x_{i}}.$$
Note this fixes one of the issues given in Example 1. In Example 1, for $k=0$, the summation gives a degree $26$ image for the coefficient of the 4 form on the numerator. The image coefficient of the 4 form on the numerator for $k=1$ for the inverse is only degree 2. The image coefficient of the 4-form on the numerator for $k=9$ of the inverse is of degree 26 which is the degree in the original Frobenius image for $k=0$. Working with low degrees for high values of $k$ makes computation slightly easier. The issue about computing the matrices and quotients still remains.

\section{Results}\label{sec2}

Up until now, in terms of coding, the fact that the singularities are of type ADE have not been used other than for the isomorphism between de Rham and rigid cohomology. The results on spectral sequences uses the assumption that the singularities were isolated weighted homogeneous. Recall from Dimca and Sticlaru [6], $H^{n}(K_{f})_{m}=\mu(X)$ for $m\geq 3(N-2)$. We call the values of $m$ such that $m\geq 3(N-2)$ the stable range.  Now, since the top diagonal for a smooth hypersurface on the $E_{1}$ page lies only in the first quadrant and the Euler characteristic is independent of whether the hypersurface is smooth or singular, we can conclude that $H^{n+1}(K_{f})_{m}=\mu(X)$. Moreover, the $E_{2}$ page is $0$ in this stable range. \\

Let $h$ be the image of Frobenius in the stable range. Suppose $h\in P_{k}\Omega_{4}$. Then find a basis for $P_{k+1}\Omega_{3}$. As we are in the stable range, there are $\mu(X)$ basis elements which we name as $\beta_{1},...,\beta_{\mu(x)}$. How we find these basis elements in high degree will be explained later. Applying the de Rham differential, $d\beta_{1}, ...,d\beta_{\mu(X)}$ lie in $P_{k}\Omega_{4}$. As the $E_{2}$ page is $0$ on the stable range, lifting $h$ back to $E_{0}$ gives $h=a_{1} d\beta_{1}+...+a_{\mu(X)}d\beta_{\mu(X)}+f_{w}h_{1}+f_{x}h_{2}+f_{y}h_{3}+f_{z}h_{4}$. Note that evaluating the equation at the singular point is an operator that annihilates the partials and gives us an equation to solve for $a_{1},...,a_{\mu(X)}$. Stetson and Baranovsky [17] showed showed that if all singularities are type $A_{1}$, we can evaluate at the singular points and can solve for the variables given. \\

To see this, suppose there exists one  $A_{1}$ singularity. Then from above, each piece of the subdiagonal is 1- dimensional meaning $g-a_{1}d\beta_{1} = b_{1}f_{w}+b_{2}f_{x}+b_{3}f_{y}+b_{4}f_{z}$. To find $a_{1}$, evaluate both sides at the singular point. Then the right hand side is $0$ by definition of a singular point. Note, plugging in any other point will give an equation but the issue is that $b_{1},b_{2},b_{3}, b_{4}$ are unknown. Therefore, we need to find operators that annihilate the right hand side to solve for $a_{1}$. Similarly, suppose there are $k$ $A_{1}$ singularity. Then $g-a_{1}d\beta_{1}-...-a_{k}d\beta_{k}= b_{1}f_{w}+b_{2}f_{x}+b_{3}f_{y}+b_{4}f_{z}$. We need to find $a_{1},...,a_{k}$; so $k$ equations are needed. Evaluation at each of the singular points will give $k$ equations. The equations will be linearly independent. In fact, I show linear independence for the general ADE case later in the paper. Hence, for $A_{1}$ singularities, finding the coefficients for de rham is simple.	\\

Suppose our hypersurface has one $ A_{2}$ singularity. Then $g-a_{1}d\beta_{1}-a_{2}d\beta_{2}= b_{1}f_{w}+b_{2}f_{x}+b_{3}f_{y}+b_{4}f_{z}$. We need to find $a_{1}$ and $a_{2}$ but evaluating at the singular point only gives $1$ equation. Where will the second equation come from? In this case, the normal form of an $A_{2}$ singularity is $uv=t^{3}$. The partials are given by $v,u,3t^{2}$. Along with evaluation at the origin, the operator given by $\frac{\partial}{\partial t}\vert _{(0,0,0)}$ annihilates the Jacobian. The idea is to transfer this operator to the original coordinates to obtain the second operator for the second equation. Hence, I establish an equality between the space annihilated by specific operators depending on our ADE singularities and the Jacobian ideal for polynomials with degree in the stable range \ref{Theorem 3.5}.

\subsection{ADE Operators}
Before we continue, in the case that there are two singularities in the same affine open set, we need an algebraic way of working locally around the singularity. 

\begin{definition} Let $M$ be a finite dimensional module over a polynomial ring $R$ in several variables over $\mathbb{C}$. Let $\tilde{R}$ be the power series ring in the variables of $R$. We define the formal completion of $M$ as $M\otimes_{R}\tilde{R}$. \end{definition} 

\begin{definition} A module over a polynomial ring in variables $(x,y,z)$ is supported at $(\alpha, \beta, \gamma)$ if $\exists N$ such that $\forall$ $k\geq N$, $(x-\alpha)^{k}M= (y-\beta)^{k}M= (z-\gamma)^{k}M=0$. \end{definition} 

\begin{proposition}\label{Proposition 3.3} Suppose $M$ is a finite dimensional module over the polynomial ring supported at the origin. Let $\tilde{M}$ be the formal completion of $M$. Then $M\cong \tilde{M}$ as $R$-modules. \end{proposition}
\begin{proposition}\label{Proposition 3.4} Suppose $M$ is a finite dimensional module over polynomial ring supported at $(\alpha, \beta, \gamma)\neq (0,0,0)$. Let $\tilde{M}$ be the formal completion of $M$. Then $\tilde{M}=0$. \end{proposition}
Assuming these two claims, formal completion is a way to study a singularity locally. Proposition \ref{Proposition 3.3} allows us to work with the polynomial ring as opposed to the power series ring. These are standard results in algebra so I will assume the reader is familiar with these results.\\

\begin{theorem}\label{Theorem 3.5} For type $A_{n}$ singularities, the space of power series in $\mathbb{C}[[u,v,t]]$ annihilated by the differential operators $$ev\vert_{(0,0,0)}, \frac{\partial }{\partial t}\vert_{(0,0,0)}, \frac{\partial ^{2}}{\partial^{2} t}\vert_{(0,0,0)}, ... \frac{\partial ^{n-1}}{\partial^{n-1} t}\vert_{(0,0,0)}$$ is equal to the Jacobian ideal for degrees in the stable range. The same results hold for type $D_{n}$ singularities with operators $$\frac{\partial}{\partial v}\vert_{s}, \frac{\partial}{\partial t}\vert_{s}, \frac{\partial^{2} }{\partial^{2} t}\vert_{s}, ..., \frac{\partial^{n-3} }{\partial^{n-3} t}\vert_{s}.$$ Similarly, the same results hold for type $E_{6}$ singularities with operators $$\frac{\partial}{\partial v}\vert_{s}, \frac{\partial}{\partial t}\vert_{s},  \frac{\partial}{\partial v}\frac{\partial}{\partial t}\vert_{s},\frac{\partial^{2} }{\partial^{2} t}\vert_{s}, \frac{\partial^{2} }{\partial^{2} t}\frac{\partial}{\partial v}\vert_{s}.$$ For $E_{7}$ singularities, the operators are $$\frac{\partial}{\partial v}\vert_{s}, \frac{\partial}{\partial t}\vert_{s}, \frac{\partial}{\partial v}\frac{\partial}{\partial t}\vert_{s},  \frac{\partial^{2} }{\partial^{2} t}\vert_{s}, \frac{\partial^{3} }{\partial^{3} t}-\frac{\partial^{2} }{\partial^{2} v}\vert_{s}, \frac{\partial^{4} }{\partial^{4} t}-3\frac{\partial^{2} }{\partial^{2} v}\frac{\partial}{\partial t}\vert_{s}.$$ Finally, the operators for $E_{8}$ singularities are given by $$\frac{\partial}{\partial v}\vert_{s}, \frac{\partial}{\partial t}\vert_{s}, \frac{\partial}{\partial v}\frac{\partial}{\partial t}\vert_{s},\frac{\partial^{2} }{\partial^{2} t}\vert_{s}, \frac{\partial^{2} }{\partial^{2} t}\frac{\partial}{\partial v}\vert_{s}, \frac{\partial^{3} }{\partial^{3} t}\vert_{s}, \frac{\partial^{3} }{\partial^{3} t}\frac{\partial}{\partial v}\vert_{s}.$$  \end{theorem} 
\begin{proposition}\label{Proposition 3.6}: The differential operator $\frac{\partial ^{k}}{\partial^{k} t}\vert_{(0,0,0)}$ is mapped to a combination of $k$ order and lower differential operators in $x,y,z$ evaluated at the origin through the analytic change of coordinates. The same holds true if we replace $t$ with $u$ or $v$. \end{proposition}

\begin{proof} Proposition 3.6 is standard Calculus result which I will assume the reader is familiar with. To prove Theorem \ref{Theorem 3.5}, first note that the space of polynomials annihilated by the differential operators contains the ideal generated by the partials. The partials are given by $u,v,t^{n}$. We proceed to prove by induction. First it is clear evaluation at the origin annihilates the partials as this is the singular point. Let $h=h_{1}u+h_{2}v+h_{3}t^{n}$ where $h_{1},h_{2},h_{3}\in \mathbb{C}[[u,v,t]]$. For simplicity, let $s$ be the origin. The product rule shows that  
	$$\frac{\partial }{\partial t}h\vert_{s}= \frac{\partial }{\partial t}h_{1}\vert_{s} \cdot u\vert_{s}+ \frac{\partial }{\partial t}h_{2}\vert_{s}\cdot v\vert_{s} + \frac{\partial }{\partial t}h_{3}\vert_{s}\cdot t^{n}\vert_{s}+ nh_{3}t^{n-1}\vert_{s} = 0$$
	Suppose that $$\frac{\partial ^{i}}{\partial^{i} t}h_{1}u\vert_{s},\frac{\partial ^{i}}{\partial^{i} t}h_{2}v\vert_{s}, \frac{\partial ^{i}}{\partial^{i} t}h_{3}t^{n}\vert_{s}=0$$ for $i=0,...,k$ where $i=0$ is the evaluation operator. Then for simplicity of notation, let $D(i,j)f_{1}f_{2}= \frac{\partial ^{i}}{\partial^{i} t}f_{1}\vert_{s}\frac{\partial ^{j}}{\partial^{j} t}f_{2}\vert_{s}$. Then for $k+1\leq n-1$
	$$ \frac{\partial ^{k+1}}{\partial^{k+1} t}h_{1}u\vert_{s}= D(k+1,0)h_{1}u+D(k,1)h_{1}u+...+D(1,k)h_{1}u+D(0,k+1)h_{1}u$$
	$D(k+1,0)h_{1}u=0$ because $u$ evaluates to $0$, and $D(0,k+1)h_{1}u=0$ since we are differentiating $u$ with respect to $t$. The other terms are $0$ by the induction hypothesis. For
	$$ \frac{\partial ^{k+1}}{\partial^{k+1} t}h_{2}v\vert_{s}= D(k+1,0)h_{2}v+D(k,1)h_{2}v+...+D(1,k)h_{2}v+D(0,k+1)h_{2}v,$$
	$D(k+1,0)h_{1}v=0$ because $v$ evaluates to $0$, and $D(0,k+1)h_{2}v=0$ as we are differentiating $v$ with respect to $t$. The other terms are $0$ by induction hypothesis. For
	$$ \frac{\partial ^{k+1}}{\partial^{k+1} t}h_{3}t^{n}\vert_{s}= D(k+1,0)h_{3}t^{n}+D(k,1)h_{3}t^{n}+...+D(1,k)h_{3}t^{n}+D(0,k+1)h_{3}t^{n},$$
	$D(k+1,0)h_{1}t^{n}=0$ because $t^{n}$ evaluates to $0$ and $D(0,k+1)h_{3}t^{n}=0$ as we are differentiating $t^{n}$ with respect to $t$ $k+1$ times. For $k+1\leq n-1$, this leads to $C\cdot t^{j}$ for some $j>0$ and $C$ a constant so evaluation at the origin gives $0$ .The other terms are $0$ by induction hypothesis.  \\
	
	Therefore, the operators stated in Theorem \ref{Theorem 3.5} annihilate any linear combination of the partials and hence, the space of polynomials annihilated by the operators contain the Jacobian ideal. \\ \\
	Let $S$ be the space annihilated by the differential operators. Define
	\begin{align*}
		&\phi: \mathbb{C}[[u,v,t]]\longrightarrow \mathbb{C}^{n}\\
		&\phi(h)= (\text{ev}(h)\vert_{s}, \frac{\partial }{\partial t}h\vert_{s}, ..., \frac{\partial ^{n-1}}{\partial^{n-1} t}h\vert_{s}).
	\end{align*}
	The kernel of $\phi$ is $S$. The map $\phi$ is surjective. Let $e_{i}$ be the vector that is $1$ on the ith component and $0$ elsewhere. Then $1$ is mapped to $e_{1}$, $t$ is mapped to $e_{2}$, $t^{2}$ is mapped to $2e_{3}$, and continuing on, $t^{n}$ is mapped to $n!e_{n}$. Hence, the map is surjective. So we have 
	$$\mathbb{C}[[u,v,t]]/S \cong \mathbb{C}^{n}.$$
	However, if $J$ is the Jacobian ideal generated by $u,v,t^{n}$, the quotient $\mathbb{C}[[u,v,t]]/J \cong  \mathbb{C}^{n}$ is the space generated by $1,t, ..., t^{n-1}$. Since $J\subset S$, we have that $J=S$.  Note the proof above works if we replace $\mathbb{C}[[u,v,t]]$ with $\mathbb{C}[u,v,t]$. Using Proposition \ref{Proposition 3.3}, since $u,v,t^{n}$ annihilate the  quotient, $\mathbb{C}[[u,v,t]]/(g_{u},g_{v},g_{t})$ is supported at $s$ so it is isomorphic to the polynomial ring. \\
	
	The same result holds for the corresponding space of differential operators in $x,y,z$ variables. We can construct $\phi$ for the operators given from Proposition 
	\ref{Proposition 3.6}. It remains to show $\mathbb{C}[[x,y,z]]/(f_{x},f_{y},f_{z})$ is supported at the origin. The analytic change of coordinates maps the origin to the origin. Hence, the change of coordinates has no constant term. Since $\mathbb{C}[[u,v,t]]/(g_{u},g_{v},g_{t})$ is supported at the origin, there exists $N$ such that 
	$$u^{N}\mathbb{C}[[u,v,t]]/(g_{u},g_{v},g_{t}), v^{N}\mathbb{C}[[u,v,t]]/(g_{u},g_{v},g_{t}), t^{N}\mathbb{C}[[u,v,t]]/(g_{u},g_{v},g_{t}) = 0.$$ 
	Let $x=h_{1}(u,v,t), y=h_{2}(u,v,t), z=h_{3}(u,v,t)$. Then by the analytic change of coordinates 
	$$x^{3N}\mathbb{C}[[x,y,z]]/(f_{x},f_{y},f_{z}), y^{3N}\mathbb{C}[[x,y,z]]/(f_{x},f_{y},f_{z}), z^{3N}\mathbb{C}[[x,y,z]]/(f_{x},f_{y},f_{z}) $$
	map to zero since the change of coordinates gives an isomorphism on the level of quotients since Jacobian ideal is invariant under change of coordinates. Hence, $\mathbb{C}[[x,y,z]]/(f_{x},f_{y},f_{z})$ is supported at the origin as well so it is isomorphic to the polynomial ring. Hence, my code can work with the polynomial ring instead of the power series ring. \\
	
	Note this characterizes the Jacobian ideal,$J$, as a set annihilated by the corresponding differential operators in the stable range, which I denoted as $S$. In the stable range, the quotient of the polynomial ring by the $J$ is of dimension $M$ where $M$ is the global Milnor number. There are $M$ differential operators which map surjectively to $\mathbb{C}^{M}$ as given by the map $\phi$ above with quotient $S$. As $J \subset S$, $J=S$. \\
	
	Note the same proof works for type $D_{n}$ and type $E_{n}$ singularities as well. The standard equation for $D_{n}$ is given by $u^{2}+tv^{2}+t^{n-1}=0$. The Jacobian ideal is $J=(u,vt, t^{n-2}+v^{2})$. The $n$ operators that annihilate any element of the Jacobian are evaluation at the origin, $$\frac{\partial}{\partial v}\vert_{s}, \frac{\partial}{\partial t}\vert_{s}, \frac{\partial^{2} }{\partial^{2} t}\vert_{s}, ..., \frac{\partial^{n-3} }{\partial^{n-3} t}\vert_{s}$$. 
	Let $S$ be the space of polynomials annihilated by all the differential operators. As in the proof of Theorem \ref{Theorem 3.5}, the space $S$ contains the Jacobian ideal, $J$. Define
	\begin{align*}
		&\phi: \mathbb{C}[[u,v,t]]\longrightarrow \mathbb{C}^{n}\\
		&\phi(h)= (\text{ev}(h)\vert_{s}, \frac{\partial }{\partial v}h\vert_{s},\frac{\partial}{\partial t}h\vert_{s}, ..., \frac{\partial^{n-3} }{\partial^{n-3} t}h\vert_{s}).
	\end{align*}
	The kernel is $S$ and the map is surjective as the polynomials $1, v, t, ...,t^{n-3}$ give a constant times vectors $e_{1},...,e_{n}$ respectively as in the proof of Theorem \ref{Theorem 3.5}. Hence, by rank nullity, we have $\mathbb{C}[[u,v,t]]/S \cong \mathbb{C}^{n}$. In the stable range, we have that the quotient by Jacobian is a space of dimension $n$, hence $S=J$.\\ 
	
	Now the standard $E_{6}$ equation is given by $u^{2}+v^{3}+t^{4}=0$. The Jacobian ideal is $J=(u,v^{2},t^{3})$.  Along with evaluation at the origin, the operators that annihilate any element of the Jacobian ideal is $$\frac{\partial}{\partial v}\vert_{s}, \frac{\partial}{\partial t}\vert_{s},  \frac{\partial}{\partial v}\frac{\partial}{\partial t}\vert_{s},\frac{\partial^{2} }{\partial^{2} t}\vert_{s}, \frac{\partial^{2} }{\partial^{2} t}\frac{\partial}{\partial v}\vert_{s}.$$ 
	Again, let $S$ be the space of polynomials annihilated by our operators and we have $J\subset S$. Define
	\begin{align*}
		&\phi: \mathbb{C}[[u,v,t]]\longrightarrow \mathbb{C}^{6}\\
		&\phi(h)= (\text{ev}(h)\vert_{s}, \frac{\partial }{\partial v}h\vert_{s},\frac{\partial}{\partial t}h\vert_{s},\frac{\partial}{\partial v}\frac{\partial}{\partial t}h\vert_{s}, \frac{\partial^{2} }{\partial^{2} t}h\vert_{s}, \frac{\partial^{2} }{\partial^{2} t}\frac{\partial}{\partial v}h\vert_{s}).
	\end{align*}
	The kernel is $S$ and the map is surjective as the polynomials $1, v, t, vt, t^{2}, vt^{2}$ map to a constant times vectors $e_{1},...,e_{6}$ respectively. Hence, $S=J$ in the stable range. \\
	The standard $E_{7}$ equation is given by $u^{2}+v^{3}+vt^{3}=0$. The Jacobian ideal is given by $J=(u, 3v^{2}+t^{3}, 3vt^{2})$ Along with evaluation at the origin, the operators that annihilate any element of the Jacobian are $$\frac{\partial}{\partial v}\vert_{s}, \frac{\partial}{\partial t}\vert_{s}, \frac{\partial}{\partial v}\frac{\partial}{\partial t}\vert_{s},  \frac{\partial^{2} }{\partial^{2} t}\vert_{s}, \frac{\partial^{3} }{\partial^{3} t}-\frac{\partial^{2} }{\partial^{2} v}\vert_{s}, \frac{\partial^{4} }{\partial^{4} t}-3\frac{\partial^{2} }{\partial^{2} v}\frac{\partial}{\partial t}\vert_{s}.$$ Let $S$ be the space of polynomials annihilated by all the differential operators. Again, we have $J\subset S$.  Define
	\begin{align*}
		&\phi: \mathbb{C}[[u,v,t]]\longrightarrow \mathbb{C}^{7}\\
		&\phi(h)= (\text{ev}(h)\vert_{s}, \frac{\partial}{\partial v}h\vert_{s}, \frac{\partial}{\partial t}h\vert_{s}, ... ,\frac{\partial^{4} }{\partial^{4} t}h-3\frac{\partial^{2} }{\partial^{2} v}\frac{\partial}{\partial t}h\vert_{s}).
	\end{align*}
	The kernel is $S$ and the map is surjective as the polynomials $1, v, t, vt,  t^{2}, t^{3}-v^{2}, t^{4}-v^{2}t$ map to a constant times vectors $e_{1},...,e_{7}$ respectively. Hence, $S=J$ in the stable range. \\
	The standard $E_{8}$ equation is given by $u^{2}+v^{3}+t^{5}$. The Jacobian ideal is given by $J=(u,v^{2},t^{4}).$ Along with evaluation at the origin, the operators that annihilate any element of the Jacobian area $$\frac{\partial}{\partial v}\vert_{s}, \frac{\partial}{\partial t}\vert_{s}, \frac{\partial}{\partial v}\frac{\partial}{\partial t}\vert_{s},\frac{\partial^{2} }{\partial^{2} t}\vert_{s}, \frac{\partial^{2} }{\partial^{2} t}\frac{\partial}{\partial v}\vert_{s}, \frac{\partial^{3} }{\partial^{3} t}\vert_{s}, \frac{\partial^{3} }{\partial^{3} t}\frac{\partial}{\partial v}\vert_{s}.$$ 
	Let $S$ be the space of polynomials annihilated by all the differential operators. We have $J\subset S$. Define 
	\begin{align*}
		&\phi: \mathbb{C}[[u,v,t]]\longrightarrow \mathbb{C}^{8}\\
		&\phi(h)= (\text{ev}(h)\vert_{s}, \frac{\partial}{\partial v}h\vert_{s}, \frac{\partial}{\partial t}h\vert_{s}, ... ,\frac{\partial^{3} }{\partial^{3} t}\frac{\partial}{\partial v}h\vert_{s}).
	\end{align*}
	The kernel is $S$ and the map is surjective as the polynomials $1, v, t, vt, t^{2}, t^{2}v, t^{3}, t^{3}v$ map to a constant times vectors $e_{1},...,e_{8}$ respectively. Hence, $S=J$ in the stable range. \\ \end{proof} 

\begin{proof}To prove Proposition \ref{Proposition 3.6}, recall the multivariable chain rule.
	$$\frac{\partial}{\partial t}h= \frac{\partial x}{\partial t}\vert_{s} \cdot  \frac{\partial }{\partial x}h\vert_{s} + \frac{\partial y}{\partial t}\vert_{s}\cdot \frac{\partial}{\partial y}h\vert_{s} + \frac{\partial z}{\partial t}\vert_{s}\cdot \frac{\partial}{\partial z}h\vert_{s}.$$
	
	Let $x=f_{1}(u,v,t), y=f_{2}(u,v,t), z=f_{3}(u,v,t) \in \mathbb{C}[[u,v,t]]$ be the analytic change of coordinates. Then $\frac{\partial x}{\partial t}\vert_{s}$ is the coefficient of $t$ in the power series of $x$. $\frac{\partial w}{\partial t}\vert_{s}$ is the coefficient of $t$ in the power series of $y$, and $\frac{\partial z}{\partial t}\vert_{s}$ is the coefficient $t$ in the power series of $z$. Hence, the operator $\frac{\partial}{\partial t}$ is a linear combination of first order operators in $x,y,z$.  \\ 
	
	What about $(\frac{\partial}{\partial t})^{2}$? This is 
	\begin{align*}
		\frac{\partial}{\partial t}\frac{\partial}{\partial t}&= \frac{\partial}{\partial t}\left(\frac{\partial x}{\partial t} \cdot  \frac{\partial }{\partial x} + \frac{\partial y}{\partial t}\cdot \frac{\partial}{\partial y} + \frac{\partial z}{\partial t}\cdot \frac{\partial}{\partial z}\right)\\
		&=\frac{\partial}{\partial t}(\frac{\partial x}{\partial t} \cdot  \frac{\partial }{\partial x}) + \frac{\partial}{\partial t}(\frac{\partial y}{\partial t}\cdot \frac{\partial}{\partial y}) + (\frac{\partial}{\partial t}\frac{\partial z}{\partial t}\cdot \frac{\partial}{\partial z}).
	\end{align*}
	I will compute the first term and the rest follow in the exact same way. In the first half of the product rule, what I want to do is take the derivative of $x$ with respect to $t$ twice and evaluate at $0$. This is equivalent to 2 times the coefficient of $t^{2}$ in the power series expansion of $x$.
	In the second half of the product rule, we have $$(\frac{\partial}{\partial t} \frac{\partial}{\partial x})\frac{\partial x}{\partial t}\vert_{(0,0,0)}=C\cdot \frac{\partial}{\partial t} \frac{\partial}{\partial x}\vert_{(0,0,0)},$$
	where $C$ is the coefficient of $t$ in the power series of $x$. Now
	\begin{align*}
		C\frac{\partial}{\partial t} \frac{\partial}{\partial x}\vert_{s}&=C(\frac{\partial x}{\partial t}\cdot  \frac{\partial }{\partial x}\frac{\partial}{\partial x} + \frac{\partial y}{\partial t}\cdot \frac{\partial}{\partial y}\frac{\partial}{\partial x} + \frac{\partial z}{\partial t}\cdot \frac{\partial}{\partial z}\frac{\partial}{\partial x})\vert_{s}\\
		&=C(C\frac{\partial }{\partial x}\frac{\partial}{\partial x}+C_{1}\frac{\partial}{\partial y}\frac{\partial}{\partial x}+C_{2}\frac{\partial}{\partial z}\frac{\partial}{\partial x})\vert_{s}
	\end{align*}
	where $C_{1}, C_{2}$ are the coefficients of $t$ in the power series of $y$ and $z$ respectively. Hence, we have a linear combination of second order partials. \\
	
	Suppose $\frac{\partial^{i}}{\partial^{i} t}\vert_{s}$ is a linear combination of $k$ order and lower differential operators in $x,y,z$ for $i$ up to $k$. So $\frac{\partial^{k}}{\partial^{k} t}=C_{0}D_{0}+C_{1}D_{1}+C_{2}D_{2}+...+C_{k}D_{k}=D$ where $C_{i}$ are constants and $D_{i}$ are differential operators of order $i$ evaluated at the origin. Then for $k\leq n$,
	\begin{align*}
		\frac{\partial^{k+1}}{\partial^{k+1} t}h\vert_{s}&= \frac{\partial}{\partial t}\frac{\partial^{k}}{\partial^{k} t}h\vert_{s}=\frac{\partial}{\partial t}\vert_{s}Dh \\
		&= \left(\frac{\partial x}{\partial t}\vert_{s} \cdot  \frac{\partial }{\partial x}h\vert_{s} + \frac{\partial y}{\partial t}\vert_{s}\cdot \frac{\partial}{\partial y}h\vert_{s} + \frac{\partial z}{\partial t}\vert_{s}\cdot \frac{\partial}{\partial z}h\vert_{s}\right)Dh.
	\end{align*}
	So $(\frac{\partial x}{\partial t}\vert_{s} \cdot  \frac{\partial }{\partial x}h\vert_{s})$ applied to $C_{i}D_{i}$ gives an order $i+1$ operator given by $\frac{\partial }{\partial x}h\vert_{s}D_{i}$. So as the highest order operator is $D_{k}$, our operator is at most order $k+1$. The same applies for the other terms. Note this actually holds in the $u,v$ variables as we just repeat the proof replacing $u$ with $t$ or $v$ with $t$. This proves Proposition \ref{Proposition 3.6}.\\ \end{proof}

\subsection{Algorithm}

Now I go over the plan for the algorithm. There will be two codes: one for computing the basis of the rigid cohomology groups and one for computing the action of Frobenius and reducing in cohomology. Use the computing basis code to compute a basis on subdiagonal $k=4N$ as this value is in the stable range. Let us call this $\beta_{1}, \beta_{2}, ..., \beta_{M}$ where $M$ is the global Milnor number. Now suppose the image of Frobenius is of degree $dN$ for some $d$. We have a basis of degree $2N$. We prove the following theorem.

\begin{theorem}\label{Theorem 3.7} Suppose we have a basis $\beta_{1}, \beta_{2}, ..., \beta_{M}$ in degree $4N$ on the subdiagonal. Suppose there exists a degree $(d-4)N$ polynomial $L$ satisfying the following properties. From Theorem 3, belonging to the Jacobian ideal is equivalent to being annihilated by specific differential operators, with evaluation at singular points being part of those operators. Assume $L$ is not annihilated by evaluation at the singular points. Furthermore, assume that lower order pieces in each term in the higher order operators annihilate $L$. Define $\chi$ to be the multiplication by $L$ map. This map is well-defined, maps an element not in the quotient to an element not in the quotient in the higher level, and the image is a basis of the higher part of the subdiagonal on the $E_{1}$ page. In other words, $L\beta_{1}, L\beta_{2}, ... L\beta_{M}$ is a basis on the higher level of our subdiagonal. \end{theorem}
\begin{proof}We will first show $\chi$ is well-defined. Since we can extend by linearity, consider the $3$-form $hdx\wedge dy \wedge dz$. Let us call the lower level on subdiagonal $B_{V}$ and the upper level on the subdiagonal $B_{U}$. Let $L$ be our multiplying factor. Then we have a map $\chi$ given by multiplication by the factor $L$.
	\begin{align*} 
		&\chi : B_{V}\longrightarrow B_{U}\\
		&\chi(\omega)= L\omega
	\end{align*}
	for a 3-form $\omega$. By linearity, suppose 
	$hdx\wedge dy\wedge dz= (f_{x}h_{1}+f_{y}h_{2}+f_{z}h_{3})dx\wedge dy \wedge dz  $. Then $$\chi(hdx\wedge dy\wedge dz)= h\cdot L dx\wedge dy \wedge dz= (f_{x}h_{1}L+f_{y}h_{2}L+f_{z}h_{3}L)dx\wedge dy \wedge dz,$$ which remains in the Jacobian. \\
	For example, if the higher order operator is $(\frac{\partial}{\partial z})^{2}+ \frac{\partial}{\partial x}\frac{\partial}{\partial y}+\frac{\partial}{\partial w}$, then the assumption is that each term in the sum annihilates $L$ and $\frac{\partial}{\partial z}, \frac{\partial}{\partial x}, \frac{\partial}{\partial y},\frac{\partial}{\partial w}$ also annihilate $L$. Then suppose that $h$ does not lie in the Jacobian ideal. We wish to show that $hL$ also does not lie in the Jacobian ideal. Suppose that
	$$hL= f_{x}h_{1}+f_{y}h_{2}+f_{z}h_{3}$$
	From Theorem 3.1, there exist operators $D_{1},...,D_{M}$ that annihilate $hL$. Since $h$ does not lie in the Jacobian, there exists $D_{i}$ that does not annihilate $h$. Applying $D_{i}$ to the right hand side gives $0$. Applying $D_{i}$ to the left hand side, by the assumption on $L$, we get $D_{i}(hL)= (D_{i}h)ev(L) \neq 0$. Hence, we have a contradiction. Thus, we can conclude the image of an element not in the image of Koszul will not be in the image of Koszul. \\
	Using the fact that elements not in the Jacobian are mapped to elements not in the Jacobian, we can now show that the image is a basis. Suppose we have linear independence. Then from the result that the dimension of the space is the global Milnor number, we immediately get that the $M$ elements span the whole space. Suppose there is a nontrivial linear combination 
	$$c_{1}L\beta_{1}+c_{2} L\beta_{2}+ ...+c_{M}L\beta_{M}=0$$
	where $0$ is a representative of an elemeent in the Jacobian as we are on the $E_{1}$ page. Then since dividing by $L$ gives a nontrivial linear combination,
	$$c_{1}\beta_{1}+...+c_{M}\beta_{M}=0.$$
	This is a contradiction since we assumed that these form a basis for the subdiagonal of degree $2N$ on the $E_{1}$ page. Hence, we must have linear independence of the new basis elements and from the argument above, these $M$ terms form a basis for the subdiagonal of degree $dN$. This proves Theorem 3.7. \\ \end{proof} 

Let $h$ be the image of Frobenius of degree $dN$. By Theorem \ref{Theorem 3.7}, we can apply the de Rham differential of the basis for the subdiagonal of degree $dN$ and call them $\alpha_{1},...,\alpha_{M}$.  Since all terms on the $E_{2}$ page are $0$ past the first quadrant, there exist $a_{1}, ..., a_{M}$ such that 
$$h-a_{1}\alpha_{1}-...-a_{M}\alpha_{M}=f_{w}h_{1}+f_{x}h_{2}+f_{y}h_{3}+f_{z}h_{4}.$$
There are $M$ variables $a_{1},...,a_{M}$ that need to be solved. From the proof above, there exist $M$ linearly independent differential operators that eliminate any element in the Jacobian. This is important since the right hand side of the equation will always be $0$ when we apply the differential operators to the equation above. From here, apply the differential operators to the equation above and solve the $M$ system of equations. The system of equations have solution because a function, $h$, being in Jacobian is equivalent to the $M$ operators annihilating $h$ when $h$ is in the stable range by Theorem \ref{Theorem 3.5}. From here, we consider $h-a_{1}\alpha_{1}-...-a_{M}\alpha_{M}$ and undo the Koszul as the element now lies in the Jacobian using Grobner basis. We continue the same way until we reach our basis on $E_{2}$ page. 

\begin{example}In the case we have a single singularity at say $[1:0:0:0]$, the operators are in the variables $x,y,z$ since we work in the affine open set. We can take $L$ to be $w^{k}$ for the appropriate power of $k$. Evaluation at $[1:0:0:0]$ does not annihilate $L$ while all the other operators annihilate $L$ since the other operators are in the variables $x,y,z$.  \end{example}
\begin{example} Suppose the singularities are the standard coordinates in the affine open set. In other words, the singularities are $[1:0:0:0], [0:1:0:0], [0:0:1:0],$ and $[0:0:0:1]$. In this case, suppose our corresponding operators have at most degree $k$. Then $L=w^{j}+x^{j}+y^{j}+z^{j}$ for $j>k$ will be a valid choice for Theorem 5. Since all operators are of degree at most $k$, applying the operators to $L$ and evaluating at the origin will annihilate $L$, and evaluating at the singular points will not annihilate $L$ by construction. For degrees lower, one will have to construct the matrix. \end{example}

Before giving an algorithm, to make calculations faster, in the case our hypersurface in $\mathbb{P}^{3}$ has ADE singularities, the subdiagonal vanishes on the $E_{2}$ page.  For the proof of the following theorem, we will use $p$ as an index rather than a prime.

\begin{theorem} The subdiagonal on the $E_{2}$ page vanishes in the case the hypersurface in $\mathbb{P}^{3}$ has only ADE singularities.\end{theorem}
\begin{proof} Following the notation of Theorem 5.3 of Dimca and Saito [4], let $z_{1},...,z_{r}$ be the singularities of $f$.
	Let $\eta_{j}$ be the 3-forms generated by the generators of $C[x,y,z]/(dh_{k})$,where $h_{k}$ is the local equation of $f$ around $z_{k}$. Let $\alpha_{h_{k},j}$ be the weight of $\eta_{j}$. Then from Theorem 5.3 of Dimca and Saito [4],
	$$\text{dim}(N_{p}^{2})\leq \# \left\{ (k,j) \vert \ \ \alpha_{h_{k},j}=\frac{p}{d}\right\},$$
	where $N^{2}$ is the subdiagonal on the $E_{2}$ page. Two points to note is the following. Since we only care about powers of $f$, we only care about $p$ being multiples of $d$. In this case, we only care when  $\alpha_{h_{k},j}=\frac{p}{d}\in \mathbb{Z}$. Second, the inequality runs through all singularities. If we can show that on each singularity the inequality shows that the dimension is 0, we are done since
	$$\# \left\{ (k,j) \vert \ \ \alpha_{h_{k},j}=\frac{p}{d}\right\}= \sum_{i} \# \left\{ j \vert \ \ \alpha_{h_{i},j}=\frac{p}{d}\right\}$$
	Let $\text{wt}(h\Omega)$ denote the weight of the form $h\Omega$. Let us first assume that our hypersurface has a type $A_{n}$ singularity. Then using notation from Theorem 5.3 of [4], in a local analytic coordinate system around our singularity, the function of the hypersurface can be written in the form $xy=z^{n+1}$. The weights of $x,y,z$ are $\frac{1}{2},\frac{1}{2},\frac{1}{n+1}$ respectively. The partials with respect to $x,y,z$ are $y,x,(n+1)z^{n}$; so the quotient 
	$\mathbb{C}[x,y,z]/(y,x,z^{n})$ is generated by $1,z,z^{2},...,z^{n-1}$  over $\mathbb{C}$. Hence the monomial basis of the quotient is given by
	$$dx\wedge dy \wedge dz,z dx\wedge dy \wedge dz,...,z^{n-1}dx\wedge dy \wedge dz.$$
	The weight of $dx\wedge dy \wedge dz$ is $$\frac{1}{2}+\frac{1}{2}+\frac{1}{n+1}= \frac{n+2}{n+1}.$$ Hence the weight of our forms are 
	$$\frac{n+2}{n+1}, \frac{n+3}{n+1}, ..., \frac{2n}{n+1}.$$
	Let us label these values by $\alpha_{i}$ respectively. For example, $\alpha_{1}=\frac{n+2}{n+1}$ and $\alpha_{2}=\frac{n+3}{n+1}$. By Dimca-Saito([4],Theorem 5.3), 
	$$\text{dim}(N_{p+d}^{2})\leq \# \left\{ k \vert \ \ \alpha_{k}=\frac{p}{d}\right\},$$
	where $N_{j}^{2}$ is the dimension of the subdiagonal on the $E_{2}$ page of degree $j$.
	From above, since the value of $\alpha_{k}$ ranges between $1$ and $2$ for all $k$, there is no way that $\alpha_{k}=\frac{p}{d}$. Hence, $\text{dim}(N_{p+d}^{2})=0$, and so the subdiagonal vanishes on the $E_{2}$ page.
	This extends to hypersurfaces with multiple $A_{n}$ singularities as it was noted that we can focus on one singularity at a time. \\
	Now suppose our hypersurface has a type $D_{n}$ singularity. Then in a local analytic system, our function can be written in the form $z^{2}+yx^{2}+y^{n-1}$. The weights of $x,y,z$ are $\frac{n-2}{2(n-1)}, \frac{1}{n-1},\frac{1}{2}$ respectively.
	The Jacobian ideal is given by $(2z,x^{2}+xy,y^{n-1})$. The quotient $\mathbb{C}[x,y,z]/(2z,x^{2}+xy,y^{n-1})$ is generated by $1,xy^{k},y^{j}$, where $k$ and $j$ run from 0 to $n-2$. The weight of $dx\wedge dy \wedge  dz$ is $\frac{2n-1}{2n-2}$. \\
	Let us consider the basis given by $y^{j} dx \wedge dy \wedge dz$.  This has weight $$\frac{2j}{2(n-1)}+\frac{2n-1}{2n-2}=1+\frac{2j+1}{2n-2},$$ which is never an integer since the numerator is odd and denominator is even.  \\
	Let us now consider the basis given by $xy^{j}dx\wedge dy \wedge dz$. This has weight $$\frac{2j}{2(n-1)}+\frac{2n-1}{2n-2}+\frac{n-2}{2(n-1)}=1+\frac{2j+n-1}{2n-2}.$$
	Now $j$ runs from $0$ to $n-2$. At $0$, the value is between $1$,and $2$, and at $n-2$, the value is between $2$ and $3$. So the only case we need to consider is whether the value can be $2$. However, the value $2$ means $p=2d$ so we are calculating the dimension of $N_{2}^{2d}$ which is not part of the first quadrant. Hence, the subdiagonal vanishes in the case our hypersurface has type $D_{n}$ singularity.\\
	Suppose the hypersurface has an $E_{6}$ singularity. Then there exists a local analytic system where the function of the hypersurface can be written in the form $x^{2}+y^{3}+z^{4}$. The weights of $x,y,z$ are $\frac{1}{2}, \frac{1}{3}, \frac{1}{4}$ respectively. The Jacobian ideal is given by $J=(2x,3y^{2}, 4z^{3})$. The quotient $\mathbb{C}[x,y,z]/(2x,3y^{2}, 4z^{3})$ is generated by $1, y, z, z^{2}. yz, yz^{2}$. The weight of $dx\wedge dy \wedge dz$ is given by $\frac{13}{12}$.  We have 
	\begin{align*}
		&\text{wt}(1dx\wedge dy \wedge dz)= \frac{13}{12}, \text{wt}(y dx\wedge dy \wedge dz)=\frac{17}{12}, \text{wt}(zdx\wedge dy \wedge dz)=\frac{16}{12},\\ &\text{wt}(z^{2}dx\wedge dy \wedge dz)=\frac{19}{12}, \text{wt}(yz dx\wedge dy\wedge dz)=\frac{20}{12}, \text{wt}(yz^{2} dx\wedge dy\wedge dz)=\frac{23}{12}.
	\end{align*}
	None are integers, so the subdiagonal vanishes. \\
	Suppose the hypersurface has an $E_{7}$ singularity. Then there exists a local analytic system where the function of the hypersurface can be written as $x^{2}+y^{3}+yz^{3}=0$. The weights of $x,y,z$ are $\frac{1}{2}, \frac{1}{3}, \frac{2}{9}$ respectively. The Jacobian ideal is given by $J=(2x, 3y^{2}+z^{3}, 3z^{2})$. The quotient  $\mathbb{C}[x,y,z]/(2x, 3y^{2}+z^{3}, 3z^{2}y)$ is generated by $1,y,z,y^{2},yz, z^{2}, y^{2}z$. We have 
	$\text{wt}(dx\wedge dy \wedge dz)= \frac{19}{18}$. Then 
	\begin{align*}
		&\text{wt}(ydx\wedge dy \wedge dz)=\frac{25}{18}, \text{wt}(zdx\ wedge dy \wedge dz)=\frac{23}{18}, \text{wt}(y^{2}dx\wedge dy \wedge dz)=\frac{31}{18},\\ &\text{wt}(yzdx\wedge dy \wedge dz)=\frac{29}{18}, \text{wt}(z^{2}dx\wedge dy \wedge dz)=\frac{23}{18}, \text{wt}(y^{2}z)=\frac{35}{18}. 
	\end{align*}
	Hence, since none are integers, the subdiagonal vanishes.
	Suppose the hypersurface has an $E_{8}$ singularity. Then there exists a local analytic system where the function of the hypersurface can be written as $x^{2}+y^{3}+z^{5}=0$. Then the weights of $x,y,z$ are $\frac{1}{2}, \frac{1}{3}, \frac{1}{5}$ respectively. The Jacobian ideal is given by $J=(2x, 3y^{2},5z^{4})$. The quotient $\mathbb{C}[x,y,z]/(2x, 3y^{2},5z^{4})$ is generated by $1,y,z, yz, z^{2}, z^{2}y, z^{3}, z^{3}y$. We have 
	$\text{wt}(dx\wedge dy \wedge dz)=\frac{31}{30}$. Then 
	\begin{align*}
		&\text{wt}(ydx\wedge dy \wedge dz)=\frac{41}{30}, \text{wt}(zdx\wedge dy\wedge dz)=\frac{37}{30}, \text{wt}(yzdx\wedge dy\wedge dz)= \frac{47}{30},\\ &\text{wt}(z^{2}dx\wedge dy\wedge dz)=\frac{43}{30}, \text{wt}(z^{2}ydx\wedge dy\wedge dz)=\frac{53}{20}, \text{wt}(z^{3}dx\wedge dy \wedge dz)=\frac{49}{30},\\
		&\text{wt}(z^{3}ydx\wedge dy\wedge dz)=\frac{59}{30}.
	\end{align*}
	None are integers so the subdiagonal vanishes. This concludes the proof. \\ \end{proof}  

\textbf{Algorithm for Computing Zeta Function}
\begin{enumerate}
	\item Calculate the basis on the $E_{2}$ page. Along with this, calculate the basis on the subdiagonal of the $E_{1}$ page in the stable range. Compute the basis on higher levels of the subdiagonal as explained in Theorem 5. 
	\item Compute the operators that annihilate the Jacobian ideal. 
	\item For each basis element, compute the image of inverse Frobenius and reduce the image into a linear combination of the basis elements.
	\item Compute the characteristic polynomial to obtain the zeta function. 
\end{enumerate}

There are two ways to find the differential operators. From Proposition \ref{Proposition 3.6}, we know how many of each order of operators we are looking for. Therefore, we can write out the action of all operators to that order and solve a system of equations. If we know the analytic change of coordiantes, we can use Chain Rule to find the operators.  I will provide examples of both methods. \\

\begin{example} Let $f(w,x,y,z)=zwx+w^{2}y+x^{3}-y^{2}x$. The partials are given by 
	\begin{align*}
		f_{w}&=zx+2wy \\ f_{x}&=zw+3x^{2}-y^{2} \\ f_{y}&=w^{2}-2yx \\ f_{z}=&wx.
	\end{align*}
	The singular point $s=[0:0:0:1]$ is of type $A_{4}$. One can check that $\frac{\partial }{\partial y}\vert_{s}$ annihilates the partials.  
	\begin{align*}
		&(\frac{\partial }{\partial y})^{2}(f_{w}h)\vert_{s}=((\frac{\partial }{\partial y})^{2}f_{w})h\vert_{s}=0\\
		&(\frac{\partial }{\partial y})^{2}(f_{x}h)\vert_{s}=((\frac{\partial }{\partial y})^{2}f_{x})h\vert_{s}= -2h(s).
	\end{align*}
	To fix this, we add $$2\frac{\partial }{\partial w}\vert_{s}.$$
	This will annihilate $f_{x}h$. Since this operator annihilates $f_{x}$, we have that 
	$$(\frac{\partial }{\partial y})^{2}+2\frac{\partial }{\partial w}\vert_{s}$$
	annihilates $f_{w}h$ and $f_{x}h$ for all $h$. Similarly, this operator annihilates $f_{y}h$ and $f_{z}h$.  \\ 
	The third order operator is 
	$$(\frac{\partial }{\partial y})^{3}+2\frac{\partial }{\partial w}\frac{\partial }{\partial y}-2\frac{\partial }{\partial x}\vert_{s}.$$ 
	Instead of showing all calculations which doesn't seem too beneficial, let me summarize what is getting fixed. Applying $(\frac{\partial }{\partial y})^{3}\vert_{s}$ to $f_{x}h$ does not annihilate $f_{x}h$. To fix this, we add in $2\frac{\partial }{\partial w}\frac{\partial }{\partial y}\vert_{s}$. This now annihilates $f_{x}h$ but does not annihilate $f_{w}h$. To fix this, we add in $-2\frac{\partial }{\partial x}\vert_{s}$.
\end{example} 
\begin{example} Let $f(w,x,y,z)=wzx+w^{3}+x^{3}-y^{2}x$. The partials are given by 
	\begin{align*}
		f_{w}&=zx+3w^{2} \\ f_{x}&=wz+3x^{2}-y^{2} \\ f_{y}&=-2yx \\ f_{z}&=wx. 
	\end{align*} 
	The singular point is $s=[0:0:0:1]$ is of type $A_{5}$. Instead of showing all the calculations, it is more helpful to explain what doesn't get annihilated and what the fix is. For first order operator, we have that $\frac{\partial }{\partial y}\vert_{s}$ annihilates all partials. \\ 
	
	For second order, we have
	$$(\frac{\partial }{\partial y})^{2}+2\frac{\partial }{\partial w}\vert_{s}.$$ 
	$(\frac{\partial }{\partial y})^{2}\vert_{s}$ does not annihilate $f_{x}h$, so we add in $2\frac{\partial }{\partial w}\vert_{s}$. \\ 
	
	For third order,  we have
	$$(\frac{\partial }{\partial y})^{3}+6\frac{\partial }{\partial y}\frac{\partial }{\partial w}\vert_{s}.$$ 
	$(\frac{\partial }{\partial y})^{3}\vert_{s}$ does not annihilate $f_{x}h$ so we add in $6\frac{\partial }{\partial y}\frac{\partial }{\partial w}\vert_{s}$. \\ 
	
	For fourth order, we have
	$$(\frac{\partial }{\partial y})^{4}+ 2 \binom{4}{2}(\frac{\partial }{\partial y})^{2}\frac{\partial }{\partial w}+4 \binom{4}{2}(\frac{\partial }{\partial w})^{2} - 24 \binom{4}{2} \frac{\partial }{\partial x}\vert_{s}.$$ 
	So $(\frac{\partial }{\partial y})^{4}\vert_{s}$ applied to $f_{x}h$ is not zero. Let us call this the error term. To fix this, applying $2 \binom{4}{2}(\frac{\partial }{\partial y})^{2}\frac{\partial }{\partial w}\vert_{s}$ gives us negative the error term + another term. So adding these two operators gets rid of the error term but we are left with another term. Now to get rid of this other term, we add $4 \binom{4}{2}(\frac{\partial }{\partial w})^{2}\vert_{s}$. This operator now annihilates $f_{x}h$ but in doing so, this operator does not annihilate $f_{w}h$. To fix this, we add in $- 24 \binom{4}{2} \frac{\partial }{\partial x}\vert_{s}$. Now, this operator annihilates any linear combination of the partials.
\end{example}
\begin{example} 
	Let $f(w,x,y,z)=zx^{2}-zwy+w^{2}x-wx^{2}$. This has one A1 singularity at $[0:0:0:1]$ and one A3 singularity at $[0:0:1:0]$. We work locally around the A3 singularity by letting $y=1$. Then let 
	$$g(w,x,z)=f(w,x,1,z)=zx^{2}-zw+w^{2}x-wx^{2},$$
	where $g$ has a singularity at the origin. The partials are given by
	\begin{align*}
		&g_{x}=2zx-2wx\\
		&g_{w}=-z+2wx-x^{2}\\
		&g_{z}=x^{2}-w.
	\end{align*}
	
	Consider the change of coordinates given by 
	\begin{align*}
		&u=-z+wx-x^{2}+x^{3}\\
		&v=w-x^{2}\\
		&t=x\cdot \sqrt[4]{1-x}.
	\end{align*}
	Let us reinterpret the derivative with respect to t in terms of our original coordinates. We have
	
	$$ \frac{\partial}{\partial t}= \frac{\partial x}{\partial t}\cdot  \frac{\partial }{\partial x} + \frac{\partial w}{\partial t}\cdot \frac{\partial}{\partial w} + \frac{\partial z}{\partial t}\cdot \frac{\partial}{\partial z}.$$
	Note that since $t^{4}=x^{4}-x^{5}$, $4t^{3}dt=(4x^{3}-5x^{4})dx$. Therefore, we have 
	$$\frac{\partial x}{\partial t}= \frac{4t^{3}}{4x^{3}-5x^{4}}=\frac{4x^{3}(1-x)^{3/4}}{4x^{3}-5x^{4}}=\frac{4(1-x)^{3/4}}{4-5x}.$$ 
	Thus our expression above is
	
	$$ \frac{\partial}{\partial t}= \frac{4(1-x)^{3/4}}{4-5x} \cdot  \frac{\partial }{\partial x} + \frac{\partial w}{\partial t}\cdot \frac{\partial}{\partial w} + \frac{\partial z}{\partial t}\cdot \frac{\partial}{\partial z}.$$	
	
	We have 
	\begin{align*}
		&\frac{\partial w}{\partial t}=2x\frac{\partial x}{\partial t}\\
		&\frac{\partial z}{\partial t}=x\frac{\partial w}{\partial t}+w\frac{\partial x}{\partial t}-2x\frac{\partial x}{\partial t}+3x^{2}\frac{\partial x}{\partial t}.
	\end{align*}
	
	What about $(\frac{\partial}{\partial t})^{2}$? This is 
	$$\frac{\partial}{\partial t}\frac{\partial}{\partial t}= \frac{\partial}{\partial t}\left(\frac{\partial x}{\partial t} \cdot  \frac{\partial }{\partial x} + \frac{\partial w}{\partial t}\cdot \frac{\partial}{\partial w} + \frac{\partial z}{\partial t}\cdot \frac{\partial}{\partial z}\right)$$
	$$=\frac{\partial}{\partial t}(\frac{\partial x}{\partial t} \cdot  \frac{\partial }{\partial x}) + \frac{\partial}{\partial t}(\frac{\partial w}{\partial t}\cdot \frac{\partial}{\partial w}) + (\frac{\partial}{\partial t}\frac{\partial z}{\partial t}\cdot \frac{\partial}{\partial z}).$$
	
	Let us calculate each of the 3 terms separately. \\
	
	\textbf{1st term} In the first half of the product rule, we want to take the derivative of $x$ with respect to $t$ twice and evaluate at $0$. This is equivalent to 2 times the coefficient of $t^{2}$ in the power series expansion of $x$. Let $s$ denote the origin. From 
	$\frac{\partial x}{\partial t}\vert_{s}=1$ and evaluation at the origin being $0$, the expansion of $x$ is given as
	$$x= (0+t+a_{2}t^{2}+...).$$
	We have that
	$$t^{4}=x^{4}-x^{5}=(t+a_{2}t^{2}+...)^{4}-(t+a_{2}t^{2}+...)^{5}.$$
	The $t^{5}$ coefficient in $x^{4}$ is $4a_{2}$ and the $t^{5}$ coefficient in $x^{5}$ is 1. Thus $a_{2}=\frac{1}{4}$, and so evaluation at $0$ gives $\frac{1}{2}$.\\
	In the second half of the product rule, we have 
	\begin{align*}
		(\frac{\partial}{\partial t} \frac{\partial}{\partial x})\frac{\partial x}{\partial t}\vert_{s}&=\frac{\partial}{\partial t} \frac{\partial}{\partial x}\vert_{s}\\ 
		\frac{\partial}{\partial t} \frac{\partial}{\partial x}\vert_{s}&=(\frac{\partial x}{\partial t}\cdot  \frac{\partial }{\partial x} + \frac{\partial w}{\partial t}\cdot \frac{\partial}{\partial w} + \frac{\partial z}{\partial t}\cdot \frac{\partial}{\partial z})\vert_{s}=(\frac{\partial}{\partial x})^{2}\vert_{s}.
	\end{align*}
	So first term gives $(\frac{\partial}{\partial x})^{2}+\frac{1}{2}\frac{\partial}{\partial x}$.\\
	
	\textbf{2nd term} Using the fact $\frac{\partial w}{\partial t}=2x\frac{\partial x}{\partial t}$,
	\begin{align*} 
		\frac{\partial}{\partial t}\left(\frac{\partial w}{\partial t} \frac{\partial}{\partial w}\right)\vert_{s}&= \left(\frac{\partial}{\partial t}\frac{\partial w}{\partial t}\right)\frac{\partial}{\partial w}\vert_{s} + \frac{\partial w}{\partial t} \left(\frac{\partial}{\partial t}\frac{\partial}{\partial w}\right)\vert_{s}\\
		&=\left(2\frac{\partial}{\partial t} x\right)\frac{\partial x}{\partial t}\frac{\partial }{\partial w}\vert_{s} + 2x \left(\frac{\partial}{\partial t}\frac{\partial x}{\partial t}\right)\frac{\partial }{\partial w}\vert_{s}+\frac{\partial w}{\partial t} \left(\frac{\partial}{\partial t}\frac{\partial}{\partial w}\right)\vert_{s}\\
		&=\left(2\frac{\partial}{\partial t} x\right)\frac{\partial x}{\partial t}\frac{\partial }{\partial w}\vert_{s} + \frac{\partial w}{\partial t} \left(\frac{\partial}{\partial t}\frac{\partial}{\partial w}\right)\vert_{s}\\
		&=2\frac{\partial }{\partial w}+\frac{\partial w}{\partial t} \left(\frac{\partial}{\partial t}\frac{\partial}{\partial w}\right)\vert_{s}=2\frac{\partial }{\partial w}.
	\end{align*}
	
	\textbf{3rd term} 
	\begin{align*}
		\frac{\partial}{\partial t}\left(\frac{\partial z}{\partial t} \frac{\partial}{\partial z}\right)\vert_{s}&= \left(\frac{\partial}{\partial t}\frac{\partial z}{\partial t}\right)\frac{\partial}{\partial z}\vert_{s} + \frac{\partial z}{\partial t} \left(\frac{\partial}{\partial t}\frac{\partial}{\partial z}\right)\vert_{s}\\
		&=\left(\frac{\partial}{\partial t}\frac{\partial z}{\partial t}\right)\frac{\partial}{\partial z}\vert_{s}=\frac{\partial}{\partial t}\frac{\partial z}{\partial t}\vert_{s}=\frac{\partial}{\partial t}\left(x\frac{\partial w}{\partial t}+w\frac{\partial x}{\partial t}-2x\frac{\partial x}{\partial t}+3x^{2}\frac{\partial x}{\partial t}\right)\vert_{s}\\
		&=\frac{\partial}{\partial t}\left(x\frac{\partial w}{\partial t}\right)\vert_{s}+\frac{\partial}{\partial t}\left(w\frac{\partial x}{\partial t}\right)\vert_{s}-\frac{\partial}{\partial t}\left(2x\frac{\partial x}{\partial t}\right)\vert_{s}+\frac{\partial}{\partial t}\left(3x^{2}\frac{\partial x}{\partial t}\right)\vert_{s}\\
		&=-2
	\end{align*}
	
	So 
	$$\left(\frac{\partial}{\partial t}\frac{\partial z}{\partial t}\right)\frac{\partial}{\partial z}\vert_{s}= -2\frac{\partial}{\partial z}$$
	Therefore, our second degree operator is $$(\frac{\partial}{\partial x})^{2}+\frac{1}{2}\frac{\partial}{\partial x}+2\frac{\partial}{\partial w}-2\frac{\partial}{\partial z}.$$ 
	Indeed, applying this operator and evaluating at the origin annihilates all the partial derivatives of $f$.
\end{example}

\subsection{Code and Run time Complexity}
Below is the link to the Sage code. There is a code for computing the basis on the $E_{2}$ page. This is fully automated. The user has to input the degrees of the Koszul map and function $f$. The second code is the image of Frobenius and reduction.
I attached videos in the README file of my code on GitHub and Zenodo. The link is : \url{https://zenodo.org/record/5810714#.Yc36q2jMJyw} \\

Note both Scott Stetson's (see [17]) and my Frobenius reduction code are not fully automated; however, I can say that both of our codes are as close to being fully automated as can be. By this, I mean that we cannot simply input the function $f$ and the prime $p$ and expect the zeta function because one of the necessary inputs for both of our codes involves using the calculation on the $E_{1}$ page. The user of the code does not have that information available at the start. Suppose the user wishes to truncate the Frobenius summation at $k=N$. Then let $M$ be the highest degree of the image of Frobenius. The user first finds a basis of the subdiagonal in the lowest level that is in the stable range. Next, in Scott's case of an ordinary double point, from Lemma of page 27 of Stetson and Baranovsky [17],  we have to find the $x_{0}$ factor to multiply by to obtain a basis on the higher degrees of the subdiagonal. We will then apply de Rham to the basis and solve a system of equations by evaluation at singular points. We will need to do this for every multiple of the degree of $f$ until we reach $M$. In the ADE case, we have to find the $L$ factor from Theorem 2.7 to obain a basis on the higher degrees of the subdigaonal.  We will then apply de Rham to the basis and solve a system of equations by applying our differential operators, which need not be only evaluation at the singular points. \\

Therefore, in Scott's case, the main necessary inputs and work for the code include the following:
\begin{enumerate}
	\item a function $f$
	\item a prime $p$
	\item finding a basis for the $E_{2}$ page
	\item undoing the Koszul differential through a Grobner basis
	\item finding a basis of the subdiagonal for all levels in the stable range
\end{enumerate}

For my case, the necessary inputs and work for the code include the following:
\begin{enumerate}
	\item a function $f$
	\item a prime $p$
	\item finding a basis of the $E_{2}$ page
	\item finding the differential operators from Theorem \ref{Theorem 3.5}.
	\item finding a basis of the subdiagonal for all levels in the stable range
\end{enumerate}

Sage has a function that immediately undoes the Koszul differential assuming the function belongs to the Jacobian. Therefore, long division from the user is not necessary. \\

In terms of computer run time, although the cohomology method is faster than the brute force method, the run time is long from my perspective. To put things in perspective, Scott Stetson used his own code to compute a quintic with prime $p=11$ in 3 months using two computers. The main issue for both codes is stated in top of page 4 and item 4 of both lists above. The Grobner basis for the hypersurfaces can be long and with the image of Frobenius having many terms, long division takes most of the memory and run time. In Scott's case, long division term by term on Mathematica takes even longer as the user has to manually input the long division as opposed to using Sage's lift function.

Each step gives $p$-adic precision of the terms of the inverse Frobenius by a factor of $p$. For example, if going up to $k=10$ in the summation gives accuracy of the entries up to $p^{4}$, going up to $k=11$ in the summation gives accuracy up to $p^{5}$. Here are examples of zeta functions of cubic hypersurfaces with ADE singularities.

\begin{center}
	\begin{tabular}{ c c c c c}
		Function & Singularity & E2 Basis & Zeta Function   \\ \\
		$zx^{2} -zwy + x^{3}$ & 1 A1, 2 A2 & $wy$ & $\frac{1}{(1-T)(1-5T)^{2}(1-25T)}$ \\ \\
		$zx^{2}-zwy+w^{2}x-wx^{2}$ & 1 A1, 1 A3 & $w^{2}, wx$ & $\frac{1}{(1-T)(1-5T)^{3}(1-25T)}$ \\ \\
		$zx^{2}-zwy+wx^{2}$ & 1 A1, 2 A2 & $wx$ & $\frac{1}{(1-T)(1-5T)^{2}(1-25T)}$  \\ \\
		$zx^{2}-zwy+wx^{2}-x^{3}$ & 2 A1, 1 A2 & $w^{2},wy$ & $\frac{1}{(1-T)(1-5T)^{3}(1-25T)}$ \\ \\
		$zwx-yw^{2}-y^{3}-wy^{2}$ & 2 A2 & $w^{2},wy$ & $\frac{1}{(1-T)(1-5T)^{2}(1+5T)(1-25T)}$  \\ \\
		$zwx-y^{3}$ & 3 A2 & no basis & $\frac{1}{(1-T)(1-5T)(1-25T)}$

	\end{tabular}
\end{center}

\begin{example}A more interesting hypersurface is the degree $4$ quartic given by 
	$$f=w^{3}x+(x+y+z)(x-y-z)(x+y+2z)(x-2y+z).$$
	This quartic has an $A_{5}$ singularity at $[0:0:1:-1]$ and $A_{2}$ singularities at $[0:1:0:-1], [0:1:3:-2], [0:5:1:-3], [0:3:2:1], [0:1:-1:0]$. Along with evaluation at the singular points, here are the operators that annihilate the Jacobian ideal. For simplicity of notation, I will only write the differential operators. Keep in mind one has to evaluate at the corresponding singular point after applying the differential operators. For the $A_{5}$ singularity, 
	\begin{align*}
		D_{1}&=\frac{\partial}{\partial w } + \frac{\partial}{\partial y }-\frac{\partial}{\partial z} \\
		D_{2}&=\frac{\partial^{2}}{\partial^{2} w } \\
		D_{3}&=\frac{\partial^{3}}{\partial^{3}w }-\frac{\partial}{\partial x} \\
		D_{4}&=	\frac{\partial^{4}}{\partial^{4} w }-4\frac{\partial}{\partial w }\frac{\partial}{\partial x }+\frac{\partial}{\partial w }\frac{\partial}{\partial y }-\frac{\partial}{\partial w }\frac{\partial}{\partial z }.
	\end{align*}
	For $[0:1:0:-1]$, the operator is
	\[ D_{5}= \frac{\partial}{\partial w }+\frac{\partial}{\partial x }-\frac{\partial}{\partial z }.\]
	For $[0:1:3:-2]$, the operator is
	\[ D_{6}= \frac{\partial}{\partial w }+\frac{\partial}{\partial x }+3\frac{\partial}{\partial y }-2\frac{\partial}{\partial z }.\]
	For $[0:5:1:-3]$, the operator is
	\[ D_{7}= \frac{\partial}{\partial w }+5\frac{\partial}{\partial x }+\frac{\partial}{\partial y }-3\frac{\partial}{\partial z }.\]
	For $[0:3:2:1]$. the operator is
	\[ D_{8}= \frac{\partial}{\partial w }+3\frac{\partial}{\partial x }+2\frac{\partial}{\partial y }+\frac{\partial}{\partial z }.\]
	For $[0:1:-1:0]$, the operator is
	\[ D_{9}=\frac{\partial}{\partial w }+\frac{\partial}{\partial x }-\frac{\partial}{\partial y }.\]
	The zeta function is of degree 6 so we will need point counts up to $\mathbb{F}_{{11}^{6}}$. For the point count of $\mathbb{F}_{{11}^{6}}$, since there are $4$ variables, we will need to brute force point count $11^{24}$ points. Assuming Sage takes a second to input the values, this would take Sage around $1.140015356\cdot 10^{21}$ days to run. Using the code, one can compute that the zeta function is given by $$(1-11x)^4(1+11x)^2.$$ 
	
	The case when $p=13$ is more interesting with complex roots. The interesting part of the zeta function, $Z(x)$, is given by  
	$$ Z(x)= 4826809x^{6}-171366x^{5}-26364x^{4}+1690x^{3}-156x^{2}-6x+1.$$
	
\end{example}

\begin{example}  A similar example is a hypersurface defined by equation 
	$$f=w^{2}y^{2}-xy^{3}+xwz^{2}+w^{2}x^{2}.$$
	
	This quartic has an $A_{7}$ singularity at $[0:0:0:1]$, $A_{3}$ singularity at $[1:0:0:0]$, and $E_{6}$ singularity at $[0:1:0:0]$. The interesting part of the zeta function, $Z(x)$ is of degree $5$. Along with evaluation at the singular points, here are the differential operators that annihilate the Jacobian ideal. 
	
	For the $A_{7}$ singularity, the operators are
	\begin{align*}
		D_{1}&=\frac{\partial}{\partial y} \\
		D_{2}&=\frac{\partial^{2}}{\partial^{2} y}\\
		D_{3}&=\frac{\partial^{3}}{\partial^{3} y}+6\frac{\partial}{\partial w}\\
		D_{4}&=\frac{\partial^{4}}{\partial^{4} y}+24\frac{\partial}{\partial y}\frac{\partial}{\partial w}\\
		D_{5}&=\frac{\partial^{5}}{\partial^{5} y}+60\frac{\partial^{2}}{\partial^{2}y }\frac{\partial}{\partial w}-240\frac{\partial}{\partial x}\\
		D_{6}&=\frac{\partial^{6}}{\partial^{6} y}+120\frac{\partial^{3}}{\partial^{3}y }\frac{\partial}{\partial w}-1440\frac{\partial}{\partial y}\frac{\partial}{\partial x}+360\frac{\partial}{\partial w}\frac{\partial}{\partial w}.
	\end{align*}
	
	For the $A_{3}$ singularity, the operators are
	\begin{align*}
		D_{7}&=\frac{\partial}{\partial z}\\
		D_{8}&=\frac{\partial^{2}}{\partial^{2} z}-\frac{\partial}{\partial x}.	
	\end{align*}
	
	For the $E_{6}$ singularity, the operators are
	\begin{align*}
		D_{9}&=\frac{\partial}{\partial y}\\
		D_{10}&=\frac{\partial}{\partial z}\\
		D_{11}&=\frac{\partial}{\partial y}\frac{\partial}{\partial z}\\
		D_{12}&=\frac{\partial^{2}}{\partial^{2}z }-\frac{\partial}{\partial w}\\
		D_{13}&=\frac{\partial^{2}}{\partial^{2} z}\frac{\partial}{\partial y}-\frac{\partial}{\partial w}\frac{\partial}{\partial y}.
	\end{align*}
	
	For $p=7$, the zeta function is given by 
	$$Z(x)=-16807x^{5} +2401x^{4}+686x^{3} -98x^{2} - 7x + 1.$$
	
	For $p=11$, the zeta function is given by 
	$$Z(x)=-161051x^{5} +14641x^{4}+1694x^{3} -154x^{2} - 11x + 1.$$
	
	The $p=11$ case is more interesting as there are 4 complex eigenvalues and 1 real eigenvalue.
	
\end{example}

\begin{example} In this example, we give a degree 9 zeta function, $Z(x)$, over 4 different primes.  Consider the hypersurface defined by equation
	$$f=-xy^{3}+w^{2}x^{2}+x^{2}z^{2}+w^{2}z^{2}.$$
	This hypersurface has 2 $A_{5}$ singularities at $[1:0:0:0]$ and $[0:0:0:1]$ and $A_{2}$ singularity at $[0:1:0:0]$. Along with evaluation at the singular points, here are the differential operators that annihilate the Jacobian. Again, as a reminder, for simplicity of notation, I will only provide the differential operators without the evaluation symbol. Keep in mind one has to evaluate at the singular points after applying the differential operators. 
	
	By symmetry, aside from evaluating at the singular points after, the operators for both $A_{5}$ singularities are the same. the operators are
	\begin{align*}
		D_{1}&=\frac{\partial}{\partial y}\\
		D_{2}&=\frac{\partial^{2}}{\partial^{2} y }\\
		D_{3}&=\frac{\partial^{3}}{\partial^{3} y}+3\frac{\partial}{\partial x}\\
		D_{4}&=\frac{\partial^{4}}{\partial^{4} y }+12\frac{\partial}{\partial x}\frac{\partial}{\partial y}.
	\end{align*}
	
	For the $A_{2}$ singularity, the operator is given by 
	\[ D_{5}=\frac{\partial}{\partial y}.\]
	
	For $p=7$, the zeta function is given by 
	$$Z(x)=-40353607x^{9} - 7411887x^{8} + 1411788x^{7} + 336140x^{6} + 14406x^{5} - 2058x^{4} - 980x^{3} - 84x^{2} + 9x + 1.$$
	
	For $p=11$, the zeta function is given by
	\begin{align*}
		Z(x)&=-2357947691x^{9} + 214358881x^{8} + 77948684x^{7} - 7086244x^{6} - 966306x^{5} + 87846x^{4} + 5324x^{3}\\
		& \quad{} - 484x^{2} - 11x + 1. \end{align*}
	
	For $p=13$, the zeta function is given by
	\begin{align*}Z(x)&=-10604499373x^{9} + 4329647673x^{8} - 637138788x^{7} + 31188612x^{6} + 1199562x^{5} - 92274x^{4} \\
		& \quad{} -14196x^{3} + 1716x^{2} - 69x + 1. \end{align*}
	
	For $p=17$, the zeta function is given by
	\begin{align*}
		Z(x)&=-118587876497x^{9} + 6975757441x^{8} + 1641354692x^{7} - 96550276x^{6} - 8519142x^{5} + 501126x^{4} \\
		& \quad{} + 19652x^{3} - 1156x^{2} - 17x + 1.\end{align*}
	
	The examples for $p=7$ and $p=13$ has 2 complex roots and 7 real roots, and examples for $p=11$ and $p=17$ has all real roots.
	
\end{example}

\begin{example} In this example, we give a degree 10 zeta function, $Z(x)$, over 2 different primes. Consider the hypersurface defined by equation 
	$$f=y^{4}+x^{2}yw+w^{2}z^{2}+yxz^{2}.$$
	This hypersurface has an $A_{5}$ singularity at $[0:1:0:0]$, a $D_{5}$ singularity at $[1:0:0:0]$, and $A_{1}$ singularity at $[0:0:0:1]$. Here are the operators that annihilate the Jacobian ideal aside from evaluation at the singular points. Since the operator for the $A_{1}$ singularity is simply evaluation, I just need to provide the operators for the $A_{5}$ and $D_{5}$ singularity. 
	
	For the $A_{5}$ singularity, the operators are 
	\begin{align*}
		D_{1}&=\frac{\partial}{\partial z}\\
		D_{2}&=\frac{\partial^{2}}{\partial^{2} z}-2\frac{\partial}{\partial w}\\
		D_{3}&=\frac{\partial^{3}}{\partial^{3} z}-6\frac{\partial}{\partial w}\frac{\partial}{\partial z}\\
		D_{4}&=\frac{\partial^{4}}{\partial^{4}z }-12\frac{\partial^{2}}{\partial^{2}z }\frac{\partial}{\partial w}+12\frac{\partial^{2}}{\partial^{2}w }+48\frac{\partial}{\partial y}.
	\end{align*}
	
	For the $D_{5}$ singularity, the operators are
	\begin{align*}
		D_{5}&=\frac{\partial}{\partial y}\\
		D_{6}&=\frac{\partial}{\partial x}\\
		D_{7}&=\frac{\partial^{2}}{\partial^{2} y}\\
		D_{8}&=\frac{\partial^{3}}{\partial^{3} y}-12\frac{\partial^{2}}{\partial^{2} x}.
	\end{align*}
	
	For $p=7$, the zeta function is given by 
	$$Z(x)=-282475249x^{10} + 5764801x^{8} + 235298x^{6} - 4802x^{4} - 49x^{2} + 1.$$ 
	This zeta function is interesting as there are no odd powers. Furthermore, all roots are of the form $a+bi$ where $a=0$ or $b=0$.
	
	For $p=11$, the zeta function is given by
	\begin{align*}
		Z(x)&=-25937424601x^{10} + 12861532860x^{9}  -2825639795x^{8} + 354312200x^{7} -24157650x^{6} + 199650x^{4}\\ & \quad{} -24200x^{3} + 1595x^{2} - 60x + 1.
	\end{align*}
	
	This zeta function is also interesting as there is no $x^{5}$ term. There are 6 real roots and 4 complex roots.	
\end{example}



\section{Conclusion}\label{sec3}

To conclude, I have generalized Scott's algorithm to point counting of ADE hypersurfaces through the action of Frobenius zeta function. Before Scott and my work, there had not been much progress in finding the zeta functions of singular hypersurfaces. Aside from the brute force point counting approach, even Lauder's deformation method with Picard Fuchs equation may not apply in the singular case. In my work, I have proved an equivalence between the Jacobian ideal and annihilation of differential operators as in Theorem \ref{Theorem 3.5}. In short, for the ADE case, to determine whether a polynomial in the stable range is in the Jacobian, there is no need for a Grobner basis. We simply apply differential operators and sees if the polynomial is annihilated by all of the differential operators. \\ 

While determining whether a polynomial in the stable range is in the Jacobian does not require the use of a Grobner basis, undoing the Koszul differential does. Therefore, the only remaining issue is the lifting of the Koszul differential using the Grobner basis which takes up most of the time for the code. From Theorem \ref{Theorem 3.5}, since belonging to the Jacobian is equivalent to being annihilated by all the differential operators, one  hypothesis is that instead of using a Grobner basis, a lift may be linked to these differential operators in the ADE case. Given polynomial $h$ in the stable range belongs to the Jacobian ideal, the lifts $q_{1},q_{2},q_{3},q_{4}$ such that 
$$h=q_{1}f_{w}+q_{2}f_{x}+q_{3}f_{y}+q_{4}f_{z}$$
may correspond to the functions that are annihilated by some differential operators but not all since being annihilated by all the differential operators means belonging to the Jacobian. \\
 
Aside from run time improvements for the code, we can move on to higher dimensions such as $\mathbb{P}^{4}$. The stable range will be larger; hence, one needs to compute the matrix for Koszul differential and the de Rham differential for lower levels. Furthermore, the subdiagonal may not vanish anymore. A second path is to extend to the other singularities in Arnold's list. In Arnold's classification of hypersurface singularities, along with ADE singularities, there are unimodal singularities. The list of unimodal singularities can be found in Hikami [9]. As the unimodal singularities still have normal forms, the theory of operators still holds. However, one needs to study the blow up of unimodal singularities and see if one can apply the isomorphism between de Rham cohomology and rigid cohomology. Similar to the $\mathbb{P}^{4}$ case, the subdiagonal need not vanish. A harder path would be to consider singularities not in Arnold's list. The definition of a Milnor number still holds there, and since there is no normal form to relate to, one has to consider a different approach as the theory of operators is no longer relevant. 

\section{Acknowledgements} 
I am extremely grateful to my advisor Professor Vladimir Barnaovsky for giving me advice on my research.  This project would not have been possible without Dr. Scott Stetson who provided the Mathematica code for the case of ordinary double points. I would also like to thank Professor Dimca Alexandru and Professor Morihiko Saito for providing me the correct resources on spectral sequences that enabled me to compute the zeta function and basis for $E_{2}$ page. I would also like to thank Hannah Knight and Chao Ming Lin for proofreading my text.

\section{Declarations}
\subsection{Funding and Competing Interests} No funding was received to assist with the preparation of this manuscript. The author has no relevant financial or non-financial interests to disclose.
\subsection{Data Available Statement}
The data sets generated and/or analyzed during the current study are available from the corresponding author on reasonable request.

\end{document}